\documentclass[11pt]{amsart}

\usepackage{amssymb, graphicx}

\usepackage{amsfonts}
\usepackage{latexsym}
\usepackage{amscd}

\allowdisplaybreaks[1]

\newtheorem{theorem}{Theorem}[section]

\newtheorem{lemma}[theorem]{Lemma}
\newtheorem{corollary}[theorem]{Corollary}
\theoremstyle{remark}
\newtheorem{remark}[theorem]{Remark}
\newtheorem{definition}[theorem]{Definition}            
\newtheorem{example}[theorem]{Example}

\begin{document}
\title[Derivations preserving quasinilpotent elements]{Derivations preserving quasinilpotent elements}

\author{J. Alaminos}
\author{M. Bre\v{s}ar}
\author{J. Extremera}
\author{\v S. \v Spenko}
\author{A.\,R. Villena}
\address{J. Alaminos, J. Extremera,  and A.~R. Villena, Departamento de An\' alisis
Matem\' atico, Facultad de Ciencias, Universidad de Granada,
 Granada, Spain} \email{alaminos@ugr.es, jlizana@ugr.es,
avillena@ugr.es}
\address{
M. Bre\v{s}ar, Faculty of Mathematics and Physics, University of Ljubljana, and Faculty of
Natural Sciences and Mathematics, University of Maribor, Slovenia}
\email{matej.bresar@fmf.uni-lj.si}
\address{\v S. \v Spenko,  Institute of  Mathematics, Physics, and Mechanics,  Ljubljana, Slovenia} 
\email{spela.spenko@imfm.si}

\thanks{The first, third and fifth author were supported
by MINECO Grant MTM2012--31755
and Junta de Andaluc\'ia Grants FQM-185 and 
P09-FQM-4911. The second and fourth author  were supported by ARRS Grant P1--0288.}

\thanks{2010 {\em Math. Subj. Class.} 46H05, 46H15, 47B47, 47B48. }

\begin{abstract}
We consider a  Banach algebra $A$  with the property that, roughly speaking, sufficiently many irreducible representations  of $A$ on nontrivial Banach spaces  do not vanish on all square zero elements. The class of Banach algebras with this property 
 turns out to be quite large --  it includes 
  $C^*$-algebras, group algebras on arbitrary locally compact groups, commutative algebras, $L(X)$ for any Banach space $X$, and various other examples.
Our main result states  that every derivation of $A$  that preserves the set of quasinilpotent elements has its range in the radical of $A$.
\end{abstract}

\maketitle
\section{Introduction}

Let $A$ be a Banach algebra.
The spectrum of an element $a$ in $A$ will be denoted by  $\sigma(a)$. By $Q=Q_A$ 
we denote the set of all quasinilpotent elements in $A$, i.e., $Q= \{q\in A\,|\, 
\sigma(q)=\{0\}\}$, and by  ${\rm rad}(A)$ we denote the (Jacobson)  
radical of $A$. Recall that ${\rm rad}(A) = \{q\in A\,|\, qA\subseteq Q\}$.

Let $d$ be a derivation of $A$. It is well-known that $d(A)\subseteq {\rm rad}(A)$ 
if $A$ is commutative; under the assumption that $d$ is continuous this was proved by 
Singer and Wermer \cite{SW}, and without this assumption considerably later by Thomas 
\cite{Th}. This result has been extended to  noncommutative algebras  in various 
directions. For instance, Le Page 
\cite{LP} proved  that $d(A)\subseteq Q$ implies  $d(A)\subseteq {\rm rad}(A)$ in 
case $d$ is an inner derivation. For a general derivation $d$ this was established 
somewhat later by Turovskii and Shulman 
\cite{ST} (and independently in \cite{MM}). In \cite{BM} it was proved that 
$d(A)\subseteq {\rm rad}(A)$ 
in case there exists $M > 0$ such that $r(d(x))\le Mr(x)$ for each $x\in A$, where 
$r(\,.\,)$ stands for 
the spectral radius. Katavolos and Stamatopoulos    \cite{KS} showed that if $d$ is an 
inner derivation implemented 
by a quasinilpotent element, then $d(Q)\subseteq Q$ implies $d(A)\subseteq {\rm rad}(A)$.

\emph{Does $d(Q)\subseteq Q$ implies $d(A)\subseteq {\rm rad}(A)$   for  an arbitrary derivation $d$ 
of $A$?} This question seems natural since 
 the condition $d(Q)\subseteq Q$ with $d$ arbitrary covers all  conditions from the preceding paragraph.
 However, in general the answer is negative since $Q$ can be $\{0\}$ 
even when $A$ is  noncommutative \cite{DT}, and in such a case  every nonzero inner derivation of $A$ 
gives rise to  a counterexample.
One is therefore forced to confine to special classes of Banach algebras.  Our main result, 
Theorem \ref{MT}, states that the answer to the above question is positive in case  $A$ has the property 
$\beta$ from  Definition \ref{def1} below. 
%that, roughly speaking, sufficiently many irreducible representations on nontrivial Banach spaces  of $A$ do not vanish on all square zero elements (). We call this the property $\beta$. 
 There are some obvious examples of algebras with this  property, say commutative algebras and the algebra $L(X)$ of
 continuous linear operators on any Banach space $X$. Our main point, however, is that the so-called algebras with the 
 property $\mathbb B$, introduced in the recent paper \cite{ABEV}, also have the property $\beta$ (Example \ref{exB}).  
 The class of  algebras for which the above question has a positive answer is therefore rather large, in particular it 
 contains   $C^*$-algebras, group algebras on arbitrary locally compact groups, and Banach algebras generated by idempotents.

\section{The property $\beta$} 

We will deal with the  class of Banach algebras having the following property.

\begin{definition} \label{def1} 
A  Banach algebra $A$ is said to have the {\em property $\beta$}  if there exists a family of 
continuous irreducible representations $(\pi_i)_{i\in I}$ of $A$ on Banach spaces $X_i$ such that 
\begin{enumerate}
\item[(a)] $ \bigcap_i \ker\pi _i = {\rm rad}(A)$.
\item[(b)]  If  $\dim X_i \ge 2$, then  there exists $q\in A$ such that $q^2=0$ and $\pi_i(q)\ne 0$.
\end{enumerate} 
\end{definition} 

\begin{example} 
Every commutative Banach algebra obviously  has the property $\beta$.
\end{example}

\begin{example} 
For every Banach space $X$, the Banach algebra 
%of all bounded linear operators  on $X$, 
$L(X)$ has the property $\beta$.
Indeed, just take $\pi=1$ and a nonzero finite rank nilpotent for $q$. 
More generally, a primitive Banach algebra with nonzero socle has the property $\beta$.
\end{example}

\begin{example} \label{exB} 
%{\em Banach algebras with the property $\mathbb{B}$.}
A Banach algebra $A$ is said to have the {\em property $\mathbb{B}$} if every
continuous bilinear map $\varphi\colon A\times A\to X$, where $X$ is
an arbitrary Banach space, with the property that for all $a,b\in A$,
\begin{equation*}%\label{pzp}
 ab=0 \ \ \Longrightarrow \ \ \varphi(a,b)=0,
\end{equation*}
necessarily satisfies
\[
\varphi(ab,c)=\varphi(a,bc) \ \ \ \ (a,b,c\in A).
\]
The class of Banach algebras with the property $\mathbb{B}$ is quite  large.
It includes $C^*$-algebras, group algebras on arbitrary locally compact groups, 
Banach algebras generated by idempotents, and topologically simple Banach algebras 
containing a nontrivial idempotent. Furthermore, this class is stable under the usual 
methods of constructing Banach algebras. For details we  refer the reader to 
 \cite{ABEV}.

We claim that 
\begin{center}
$A$ has the property $\mathbb B \ \Longrightarrow \ A$ has the property $\beta$.
\end{center}
Indeed, take a  continuous irreducible representation $\pi$ of a Banach algebra $A$ with 
the property $\mathbb B$ on a Banach space $X$ with $\dim(X)\ge 2$. It is enough to show 
that there exist   $a,b\in A$ such that 
\begin{equation*}%\label{pz}
ab=0, \ \pi(a)\ne 0, \ \pi(b)\ne 0.
\end{equation*} 
Namely, since $\pi(A)$ is
a prime algebra, we can  find $c\in A$ such that $\pi(b)\pi(c)\pi(a)\ne 0$.
Hence   $q=bca$ satisfies $q^2=0$ and $\pi(q)\ne 0$, as required in  Definition \ref{def1}. Assume, therefore,  that such $a$ and $b$ do not exist. That is, for all 
 $a,b\in A$,  $ab=0$ implies
$\pi(a)=0$ or $\pi(b)=0$. Then the continuous bilinear mapping 
\[
\varphi\colon A\times A\to L(X)\widehat{\otimes} L(X), \ \
\varphi(a,b)=\pi(a)\otimes\pi(b) \ \ (a,b\in A)
\]
satisfies the condition $
 ab=0 \Longrightarrow \varphi(a,b)=0$. Here $\widehat{\otimes}$ stands for the projective tensor product.
Consequently, we have
\[
\pi(a)\pi(b)\otimes\pi(c)=\pi(a)\otimes\pi(b)\pi(c) \ \ \ \ (a,b,c\in A).
\]
Let $\xi,\zeta\in X\setminus\{0\}$. There exist $a,b\in A$ such that
$\pi(a)\xi=\zeta$ and $\pi(b)\zeta=\xi$. Then
$\pi(a)\pi(b)\otimes\pi(a)=\pi(a)\otimes\pi(b)\pi(a)$ and
both $\pi(a)$ and $\pi(b)\pi(a)$ are different from zero.
This implies that there exists $\lambda\in\mathbb{C}$ such
that $\pi(a)=\lambda\pi(b)\pi(a)$. Hence
\[
\zeta=\pi(a)\xi=\lambda\pi(b)\pi(a)\xi=\lambda\pi(b)\zeta=\lambda\xi.
\]
From this  we conclude that $\dim(X)=1$, a contradiction.
\end{example}

\begin{example}\label{red}
Let $A$ have the property $\beta$ and let $(\pi_i)_{i\in I}$ be the corresponding representations.
The following constructions will be used later.
\begin{enumerate}
\item
The quotient Banach algebra $A/\text{\rm rad}(A)$ also has the property $\beta$.
 Indeed, for every $i\in I$ the representation $\pi_i$ 
drops to an irreducible representation $\varpi_i$ of the quotient Banach 
algebra $A/\text{\rm rad}(A)$ on $X_i$ by defining
\[
\varpi_i(a+\text{\rm rad}(A))=\pi_i(a) \ \ \ \ (a\in A).
\] 
It is clear that $(\varpi_i)_{i\in I}$ satisfies the required properties.
\item
Assume that $A$ does not have  an identity element. 
Let $A_\mathbf{1}$ be the Banach algebra formed by adjoining an identity to 
$A$, so that $A_\mathbf{1}=\mathbb{C}\mathbf{1}\oplus A$. For every $i\in I$,
the representation $\pi_i$ lifts to an irreducible representation $\varpi_i$ 
of $A_\mathbf{1}$ on $X_i$ by defining
\[
\varpi_i(\alpha\mathbf{1}+a)\xi \ = \ \alpha\xi+\pi_i(a)\xi \ \ \ \ 
(\alpha\in\mathbb{C}, \, a\in A, \, \xi\in X_i).
\]
Further, we adjoin the $1$-dimensional representation
$\varpi(\alpha\mathbf{1}+a)=\alpha$ $(\alpha\in\mathbb{C}, a\in A)$
to the family $(\varpi_i)_{i\in I}$. Then the resulting family satisfies the
requirements of Definition~\ref{def1}. That is,  $A_\mathbf{1}$ has the property $\beta$.
\end{enumerate}
\end{example}

\section{Tools}

The purpose of this section is to gather together the results needed for the proof of Theorem \ref{MT} below.
We start with a simple  lemma which indicates that  it is enough  to consider the condition $d(Q)\subseteq Q$ on
semisimple Banach algebras.

\begin{lemma}\label{ln}
Let $A$ be a Banach algebra and let $d$ be a derivation of $A$ 
such that $d(Q)\subseteq Q$. Then $d(\text{\rm rad}(A))\subseteq \text{\rm rad}(A)$
and the derivation $D$ of the semisimple Banach algebra $A/\text{\rm rad}(A)$,  
defined by $D(x+\text{\rm rad}(A))=d(x)+\text{\rm rad}(A)$, satisfies 
$D(Q_{A/\text{\rm rad}(A)})\subseteq Q_{A/\text{\rm rad}(A)}$.
\end{lemma}

\begin{proof}
Write $\mathcal{R}$ for $\text{rad}(A)$.  Then
$\bigl(d(\mathcal{R})+\mathcal{R}\bigr)/\mathcal{R}$ is a two-sided ideal of the
semisimple Banach algebra $A/\mathcal{R}$. 
Since $d(Q)\subseteq Q$, it follows that $d(\mathcal{R})\subseteq Q$ and so
$\bigl(d(\mathcal{R})+\mathcal{R}\bigr)/\mathcal{R}$ consists of quasinilpotent
elements of $A/\mathcal{R}$. Therefore 
$\bigl(d(\mathcal{R})+\mathcal{R}\bigr)/\mathcal{R}=\{0\}$, that is,
 $d(\mathcal{R})\subseteq\mathcal{R}$.

On account of \cite[Proposition 1.5.29(i)]{Dal}, we have $Q_{A/\mathcal{R}}=Q_A/\mathcal{R}$
and this clearly implies that $D(Q_{A/\mathcal{R}})\subseteq Q_{A/\mathcal{R}}$.
\end{proof}

We need two standard results on derivations on Banach algebras (see, e.g., 
\cite[Proposition 2.7.22(ii) and Theorem 5.2.28(iii)]{Dal}).

\begin{theorem}\label{ss}
Let $d$ be a derivation on a Banach algebra $A$. 
\begin{enumerate}
\item (Sinclair) If  $d$ is continuous, then $d(P)\subseteq P$ for each primitive ideal $P$ of $A$.
\item (Johnson and Sinclair) If $A$ is semisimple, then $d$ is automatically continuous.
\end{enumerate}
\end{theorem}

Our main tool is the Jacobson density theorem together with its extensions. First we state a version of this theorem which includes Sinclair's generalization involving invertible elements (see, e.g., \cite[Theorem 4.2.5, Corollary 4.2.6]{Ab}).

\begin{theorem}\label{JDT}
Let $\pi$ be a continuous  irreducible representation of a unital Banach algebra $A$ on a Banach space $X$. 
If $\xi_1,\ldots,\xi_n$ are linearly independent elements in $X$, and $\eta_1,\ldots,\eta_n$ are 
arbitrary elements in $X$, then there exists $a\in A$ such that $\pi(a)\xi_i = \eta_i$, $i=1,\ldots,n$. 
Moreover, if  $\eta_1,\ldots,\eta_n$  are linearly independent, then $a$ can be chosen to be invertible.
\end{theorem}

The next theorem is basically \cite[Theorem 4.6]{BB}, but stated in the analytic setting (alternatively, one can use \cite[Theorem 3.6]{BS} together with Theorem \ref{ss}).

\begin{theorem}\label{dd}
Let $d$ be a continuous derivation on  a Banach algebra $A$, and let
 $\pi$ be a continuous irreducible representation of $A$ on a Banach space $X$. The following statements are equivalent:
 \begin{enumerate}
 \item[(i)] There does not exist a continuous linear operator $T:X\to X$ 
such that $\pi(d(x))= T\pi(x) - \pi(x)T$  for all $x\in A$.
 \item[(ii)] If $\xi_1,\ldots,\xi_n$ are linearly independent elements in $X$, and $\eta_1,\ldots,\eta_n,\zeta_1,\ldots,\zeta_n,$ are arbitrary elements in $X$, then there exists $a\in A$ such that $$\pi(a)\xi_i = \eta_i\quad\mbox{and}\quad \pi(d(a))\xi_i = \zeta_i,\,\,\, i=1,\ldots,n.$$ 
 \end{enumerate}
\end{theorem}

\section{Main theorem}

We now have enough information to prove the main result of the paper.

\begin{theorem}\label{MT}
Let $A$ be a Banach algebra with the property $\beta$,   and let $Q$ be the set of its quasinilpotent elements.
If a derivation $d$ of $A$ satisfies $d(Q)\subseteq Q$, then $d(A)\subseteq {\rm rad}(A)$.
\end{theorem}

\begin{proof}
We first assume that $A$ is semisimple and has an identity element.
Obviously $d(\mathbf{1})=0$.
On account of Theorem \ref{ss}, $d$ is continuous and leaves the
primitive ideals of $A$ invariant.

Take an irreducible representation $\pi$ of $A$ on a Banach space $X$ such as in Definition \ref{def1}. 
We have to show that $\pi(d(A))=\{0\}$.

Suppose first that $\dim X= 1$. Then  $P=\ker \pi$ has codimension 1 in $A$, 
so that $A=\mathbb{C}\mathbf{1}\oplus P$. Hence
$d(A)\subseteq P$, which gives $\pi(d(A))=\{0\}$.

We now assume that $\dim X\ge 2$. According to Definition \ref{def1},
 there exists $q\in A$ such that $q^2=0$ and $\pi(q)\ne 0$. Let $\rho\in X$ be such that 
\[
\omega:=\pi(q)\rho\ne 0.
\] 
Note that $\omega$ and $\rho$ are linearly independent for $\pi(q)^2=0$. Also,
\[
\pi(q)\omega =0.
\]   
We now consider two cases. 

{\bf Case 1}. Let us first consider the possibility where conditions of  Theorem \ref{dd} are fulfilled. Then there exists $a\in A$ such that
\[
\pi(a)\rho = 0,\,\, \pi(a)\omega =0,\,\, \pi(d(a))\rho = \omega,\,\,\pi(d(a))\omega = -\rho +  \pi(d(q))\rho,
\]
and 
\[
\pi(a)\pi(d(q))\rho =0
\]
(if $\pi(d(q))\rho$ lies in the linear span of $\rho$ and $\omega$, 
then this follows from the first two identities). Note that for any $n\ge 2$,
\[
\pi(d(a^n))\rho = \pi(d(a))\pi(a)^{n-1}\rho + \cdots + \pi(a)^{n-1}\pi(d(a))\rho = 0,
\]
and, similarly,
\[
\pi(d(a^n))\omega =0.
\]
Both formulas trivially also hold for $n=0$. Consequently,
\[
\pi(d(e^a))\rho = 
\pi\Bigl(d\Bigl(\sum_{n=0}^\infty \frac{1}{n!} a^n \Bigr)\Bigr)\rho =  
\sum_{n=0}^\infty \frac{1}{n!}  \pi(d (a^n ))\rho =\pi(d(a))\rho = \omega.
\]
Similarly,
\[
\pi(d(e^a))\omega =\pi(d(a)) \omega =  -\rho +  \pi(d(q))\rho.
\]

By assumption, $d(e^{-a}qe^a)\in Q$, and hence also $e^ad(e^{-a}qe^a)e^{-a}\in Q$. Expanding $d(e^{-a}qe^a)$ according to the derivation law, and also using $e^a d(e^{-a}) + d(e^a)e^{-a} =d(\mathbf{1})=0$, 
it follows that 
\[
b:=-d(e^a)e^{-a}q + d(q) + qd(e^a)e^{-a}\in Q.
\]
However, 
\begin{align*}
\pi(b)\rho =& -\pi(d(e^a))\pi(e^{-a})\pi(q)\rho + \pi(d(q))\rho + \pi(q)\pi(d(e^a))\pi(e^{-a})\rho\\
 =  &  -\pi(d(e^a))\pi(e^{-a})\omega + \pi(d(q))\rho +  \pi(q)\pi(d(e^a))\rho\\
 =  &  -\pi(d(e^a))\omega + \pi(d(q))\rho +   \pi(q)\omega\\
 = & \rho,
\end{align*}
implying that $1\in \sigma(\pi(b))\subseteq \sigma(b)$ -- a contradiction. 
This first possibility therefore cannot occur.

{\bf Case 2}. We may now assume that  there exists a continuous linear operator $T:X\to X$ such that 
\[
\pi(d(x))= T\pi(x) - \pi(x)T
\]  
for each $x\in A$. 
Suppose there exists $\xi\in X$ such that $\xi$ and $\eta:=T\xi$ are linearly independent.  
By Theorem \ref{JDT} then  there is an invertible $a\in A$ such that $\pi(a)\rho = -\eta$ 
and $\pi(a)\omega = \xi$. Put $c:=d(aqa^{-1})$. Note that $c\in Q$ since  $aqa^{-1}\in Q$. 
However,
\begin{align*}
 \pi(c)\xi
 =&  \bigl(T\pi(a)\pi(q)\pi(a)^{-1} - \pi(a)\pi(q)\pi(a)^{-1}T \bigr)\xi\\
 =& T\pi(a)\pi(q)\omega -  \pi(a)\pi(q)\pi(a)^{-1}\eta\\
 =& \pi(a)\pi(q)\rho
 =\pi(a) \omega = \xi,
\end{align*}
and hence $1\in \sigma( \pi(c))\subseteq \sigma (c)$. 
This is a contradiction, so $T\xi$ and $\xi$ are linearly dependent for every $\xi\in X$. 
It is easy to see that this implies that $T$ is a scalar multiple of the identity, whence 
$\pi(d(A))=0$. 

Finally, we consider the case when $A$ is an arbitrary Banach algebra.
On account of Lemma~\ref{ln}, $d(\text{\rm rad}(A))\subseteq\text{\rm rad}(A)$
and therefore $d$ drops to a derivation $D$ on the semisimple Banach algebra $A/\text{\rm rad}(A)$
with the property that $D(Q_{A/\text{\rm rad}(A)})\subseteq Q_{A/\text{\rm rad}(A)}$.
According to Example~\ref{red},  $A/\text{\rm rad}(A)$ has the property $\beta$. If this Banach
algebra already has an identity element, then we apply what has previously been proved
to show that $D(A/\text{\rm rad}(A))=\{0\}$ and hence that $d(A)\subseteq\text{\rm rad}(A)$.
If  $A/\text{\rm rad}(A)$ does not have  an identity element, then we consider its unitization $B$
(considered in Example~\ref{red}) and we  extend $D$ to a derivation $\Delta$ of $B$ by 
defining $\Delta(\mathbf{1})=0$. It is clear that $Q_B=Q_{A/\text{\rm rad}(A)}$. Therefore 
$\Delta(Q_B)\subseteq Q_B$.
We thus get $\Delta(B)=\{0\}$, which implies that $D(A/\text{\rm rad}(A))=\{0\}$ 
and therefore that $d(A)\subseteq\text{\rm rad}(A)$.
\end{proof}

\begin{remark}
From the proof of Theorem \ref{MT} it is evident that in the case where $A$ is 
semisimple,  the assumption that $d(Q)\subseteq Q$ can be replaced by the milder 
assumption that $d(q)\in Q$ for every square zero element $q\in A$.
\end{remark}

\begin{corollary}
Let $A$ be a $C^*$-algebra and let $Q$ be the set of its quasinilpotent elements. If a derivation $d$ of $A$ satisfies 
$d(Q)\subseteq Q$, then $d =0$. 
\end{corollary}

\begin{corollary}
Let $G$ be a locally compact group and let $Q$ be the set of the quasinilpotent elements of  $L^1(G)$. 
If a derivation $d$ of  $L^1(G)$ satisfies $d(Q)\subseteq Q$, then $d =0$. 
\end{corollary}

\end{document}